\documentclass[12pt]{amsart}
\usepackage[active]{srcltx}
\usepackage{verbatim}
\usepackage{epsfig,graphicx,color,mathrsfs}
\usepackage{graphicx}
\usepackage{amsmath,amssymb,amsthm,amsfonts}
\usepackage{amssymb}
\usepackage[english]{babel}
\usepackage{epsfig,graphicx,color,mathrsfs}
\usepackage[english]{babel}
\usepackage[left=2.7cm,right=2.7cm,top=3cm,bottom=3cm]{geometry}

\newcommand{\N}{{\mathbb N}}

\newcommand{\R}{{\mathbb R}}

\newcommand{\rn}{{\mathbb{R}^N}}

\newcommand{\la}{\lambda}

\numberwithin{equation}{section}
\newtheorem{theorem}{Theorem}[section]
\newtheorem{proposition}[theorem]{Proposition}
\newtheorem{lemma}[theorem]{Lemma}

\newtheorem{definition}[theorem]{Definition}
\newtheorem{remark}[theorem]{Remark}

\theoremstyle{definition}

\renewcommand{\dfrac}{\displaystyle\frac}
\newcommand{\brm}{\begin{remark}\rm}
\newcommand{\erm}{\end{remark}}
\newcommand{\brms}{\begin{remark}\rm}
\newcommand{\erms}{\end{remark}}
\newcommand{\bte}{\begin{theorem}}
\newcommand{\ete}{\end{theorem}}
\newcommand{\bpr}{\begin{proposition}}
\newcommand{\epr}{\end{proposition}}
\newcommand{\ble}{\begin{lemma}}
\newcommand{\ele}{\end{lemma}}
\newcommand{\beq}{\begin{equation}}
\newcommand{\eeq}{\end{equation}}
\newcommand{\bdm}{\begin{displaymath}}
\newcommand{\edm}{\end{displaymath}}
\numberwithin{equation}{section}

\newcommand{\bos}{\begin{remark}\rm}
\newcommand{\eos}{\end{remark}}

\newcommand{\ben}{\begin{enumerate}}
\newcommand{\een}{\end{enumerate}}

\newcommand{\e }{\varepsilon }

\newcommand{\be}{\begin{equation}}
\newcommand{\ee}{\end{equation}}

\title[Radial symmetry and applications ]{Radial symmetry and applications  for a problem involving the $-\Delta_p(\cdot)$ operator and critical nonlinearity in~$\mathbb{R}^N$}

\author[L. Damascelli]{Lucio Damascelli$^+$}

\author[S.\ Merch\'an]{Susana Merch\'an$^*$}

\author[L.\ Montoro]{Luigi Montoro$^*$}

\author[B.\ Sciunzi]{Berardino Sciunzi$^*$}


\thanks{$^+$Dipartimento di Matematica, Universit\`{a} di Roma ``Tor Vergata",
V. della Ricerca Scientifica 1  Roma, Italy.
E-mail: {\em damascel@mat.uniroma2.it}}

\thanks{$^*$Dipartimento di Matematica e Informatica,
Universit\`a della Calabria,
Ponte Pietro Bucci 31B, I-87036 Arcavacata di Rende, Cosenza, Italy,
E-mail: {\em merchan@mat.unical.it}, {\em montoro@mat.unical.it}, {\em sciunzi@mat.unical.it}}

\thanks{BS  were partially supported by ERC-2011-grant: \emph{Elliptic PDE's and symmetry of interfaces and layers for odd nonlinearities.}}

\thanks{LM and BS  were partially supported by PRIN-2011: {\em Variational and Topological Methods in the Study of Nonlinear Phenomena}}

\begin{document}

\begin{abstract}
We consider weak non-negative solutions to the critical $p$-Laplace equation in $\mathbb{R}^N$
\begin{equation}\nonumber
-\Delta_p u =u^{p^*-1}\,,
\end{equation}
 in the singular case $1<p<2$. We prove that if
 $p^*\geqslant2$  then
 all the solutions  in ${\mathcal D}^{1,p}(\R^N)$ are radial (and radially decreasing) about some point.
\end{abstract}

\maketitle

\tableofcontents

\medskip
\section{Introduction}\label{introdue}
We consider weak non-negative  solutions to the critical $p$-Laplace equation in $\mathbb{R}^N$
\begin{equation}\label{eq:p}
u \geq 0 \, , \quad u\in \mathcal {D}^{1,p}(\R^N) \, , \quad-\Delta_p u =u^{p^*-1}\quad \text{in }\,\,\mathbb{R}^N
\end{equation}
where $1< p<2\leq N$,  $p^*$ is the critical exponent for the Sobolev immersion and we assume that
the nonlinearity in \eqref{eq:p} is locally Lipschitz continuous in $[0, + \infty )$, namely $p^*\geqslant2$. \\
We assume that
$u\in \mathcal {D}^{1,p}(\R^N)$ which is a space naturally associated to
the critical equation \eqref{eq:p}.
 Let us recall that
\[
\mathcal {D}^{1,p}(\R^N) = \Big\{u\in L^{p^*}(\mathbb{R}^N) \,:\, \int_{\mathbb{R}^N}|\nabla u|^p <\infty\Big\}\,,
\]
is the completion of $ C^{\infty}_{c}(\mathbb{R}^N)$ with respect to the norm $\|\varphi\|_{{D}^{1,p}(\mathbb{R}^N)}\,:=\,(\int_{\mathbb{R}^N}|\nabla \varphi|^p)^{1/p}$. \par

\begin{remark}\label{dhfgflxvvjdf}
Any solution $u\in \mathcal {D}^{1,p}(\R^N)$ of  \eqref{eq:p} belongs to  $ L^\infty(\mathbb{R}^N)$. \par
This follows by exploiting the technique of \cite{trudi} (that goes back to \cite{moser}) and then the $L^\infty$ estimates in \cite{localser}.
 The adaptation of this scheme to the quasilinear  case is straightforward (following  \cite{trudi} we obtain that
 $\|u\|_{L^{\beta p^*}(B(x,r))}\leqslant C(r)$, with $C(r)$  not depending on $x$, for some $\beta>1$, in fact we can take any $\beta$ such that $1\leqslant \beta <p^*/p$ ; \  the $L^\infty$ estimate then follows by \cite[Theorem 1]{localser}), see e.g.  \cite{peral} for the details.
 \end{remark}

Taking into account Remark \ref{dhfgflxvvjdf}, by standard $C^{1,\alpha}$ estimates (see \cite{Di,mingione,Li,Te,T}), we deduce that $u$ is locally of class  $C^{1,\alpha}$.
 This fact allows the use of the strong maximum principle (see \cite{V}) to deduce that any nonnegative nontrivial  solution to \eqref{eq:p} is actually strictly positive. \par
 So throughout the paper we can and will assume that a nonnegative nontrivial solution $u$ to \eqref{eq:p}  satisfies

 \[ u >0 \text{ in } \R^N \, , \quad
u\in \mathcal {D}^{1,p}(\R^N) \cap L^\infty(\mathbb{R}^N)  \cap C^{1,\alpha}_{loc}(\mathbb{R}^N)\,.
\]

\medskip

It is a key and open problem  to understand if, under our assumptions,  all the solutions to \eqref{eq:p} are radial and radially decreasing.\par

\medskip
  In the semilinear case $p=2$, with $N\geq 3$, the situation is well understood and
it is well known that classical positive solutions to \eqref{eq:p} are radial and radially decreasing. \\
This was first proved in \cite{GNN2} via the \emph{Moving Plane Method} (see \cite{S} and  \cite{GNN}) under suitable decay assumptions on the solution.
 Later, in \cite{CGS}, \cite{chen}, the same result has been obtained, without any a-priori decay assumption, using the Kelvin transform. Consequently, the solutions to \eqref{eq:p} in the case $p=2$ can be classified (see \cite{Ta}) up to translations and scale invariance. The classification of the solutions had a great impact in the literature and
  it has been exploited in many issues such as a-priori estimates, blow-up analysis and asymptotic analysis, providing
 striking results. We refer to \cite{chen} for the two dimensional case.\\

 In the general quasilinear case the problem is still open, and we provide a partial answer in this paper in the singular case $1<p<2$ assuming that the nonlinearity $u^{p^*-1}$ is locally Lipschitz continuous in $[0, + \infty )$, which is the case for $p^*\geqslant 2$.\\

\noindent The fact that the problem in the quasilinear setting is  difficult to study depends not only on the lack of general comparison principles for quasilinear operators, but also on the fact that a Kelvin type transformation is not available in this case.\\
 Let us mention that in \cite{DPR, DR} (see also \cite{Szou}) symmetry and monotonicity results have been obtained, in the spirit of \cite{GNN2}, under suitable a-priori assumptions on the solutions and/or the nonlinearity. In the applications such assumptions  are not generally easy to check and this is the reason for which  many issues in the context of a-priori estimates, blow-up analysis and asymptotic analysis are still not sufficiently understood.\\

 \noindent Here we consider  $1<p<2$ with $p^*\geqslant2$ and  we prove, exploiting
    the \emph{Moving Plane Method} \cite{S} and Sobolev's inequality (and a technique introduced in \cite{DP}  to avoid ``local symmetries''), the following

\begin{theorem}\label{radialtheorem}
Let $u\in{D}^{1,p}(\R^N)$ be a (nonnegative and nontrivial, hence) positive solution to \eqref{eq:p} and assume that
 $1< p<2$ and $p^*\geqslant 2$, namely $\frac{2N}{N+2}\leqslant p<2$. \\
 Then  $u$ is radial with respect to some point $x_0\in\mathbb{R}^N$ which is the only critical point of $u$ and it is strictly radially decreasing, i.e. there exists a   $C^1$ function $v:[0,+\infty)\rightarrow (0,+\infty)$ such that $v' (r) <0 $ for any $r>0$ and such that
$u(x)=v(r)$, \ $ r=|x-x_0|$.
\end{theorem}

 \noindent Remarkably, this allows us  to exploit the classification results in \cite{bidaut,ClassGV} and deduce that, for $\frac{2N}{N+2}\leqslant p<2$, all the solutions to \eqref{eq:p} belonging to ${D}^{1,p}(\R^N)$ are given by the following expression:
\begin{equation}\label{classificationveron}
\mathcal{U}_{\lambda,x_0}:=\left[\frac{\lambda^{\frac{1}{p-1}}(N^{\frac{1}{p}}(\frac{N-p}{p-1})^{\frac{p-1}{p}})
}{
\lambda^{\frac{p}{p-1}}+|x-x_0|^{\frac{p}{p-1}}
}\right]^{\frac{N-p}{p}},\quad \lambda>0\qquad x_0\in\mathbb{R}^N.
\end{equation}
Note that, by \cite{Ta}, it follows that the family of functions given by \eqref{classificationveron} are minimizers
to
\begin{equation}\label{SSS}
S:=\,\min_{\underset{\varphi\neq 0}{\varphi\in \mathcal {D}^{1,p}(\R^N)} }\dfrac{\int_{\mathbb{R}^N}|\nabla\varphi|^pdx}{\left(\int_{\mathbb{R}^N}\varphi^{p^*}dx\right)^{\frac{p}{p^*}}}\,.
\end{equation}

The adaptation of the moving plane procedure to the case of equations in bounded domains involving the $p$-Laplacian  in the singular
case $1<p<2$ was carried out in \cite{lucio,DP}. The technique  developed in \cite{DP} that allows to  pass from local symmetry results to symmetry results  was then exploited in  the cited papers  \cite{DPR, DR} to prove symmetry results in $\R^N$, again in the singular
case $1<p<2$ . We will also exploit and explain this technique in Section 2.\\
 The degenerate case $p>2$  in bounded smooth domains and positive nonlinearities
 was considered in \cite{DS1},  \cite{DS2} using extensively  weighted Poincar\'{e}'s  type inequalities, that are not available in unbounded domains.
The problem in half spaces has been studied in \cite{DS3,FMS1,FMS}.\\

\noindent We point out also some first consequences of our symmetry result. \\
Once that Theorem \ref{radialtheorem} is available, exploiting the abstract results in  \cite{willem},
we can recover the compactness result  of M. Struwe \cite{struwe}, for the case of $p$-Laplace equations.\\
 To state this result, for $\varphi \in  {W}^{1,p}(\Omega)$,  let us define first the energy functional
\begin{equation}\nonumber
J(\varphi)=\frac 1p\int_{\Omega}|\nabla \varphi|^pdx-\frac{1}{p^*}\int_{\Omega}\varphi^{p^*}dx\,,
\end{equation}
and for  $\varphi \in  \mathcal {D}^{1,p}(\R^N)$ let us also define,
\begin{equation}\nonumber
J_\infty(\varphi)=\frac 1p\int_{\R^N}|\nabla \varphi|^pdx-\frac{1}{p^*}\int_{\R^N}\varphi^{p^*}dx\,.
\end{equation}
As a consequence of our symmetry result, we have the following
\begin{theorem}\label{triscorojfmdmgmgmgbmbxbdkjgdfkghk}
Let $1<p<2$ and assume that $p^*\geqslant 2$, namely $\frac{2N}{N+2}\leqslant p<2$.
Let $\{u_\e\}\subset W^{1,p}_0(\Omega)$ be
such that
\begin{center}
$J(u_\varepsilon)\underset{\varepsilon\rightarrow 0}{\longrightarrow} c\quad$ and
 $\quad J'(u_\varepsilon)\underset{\varepsilon\rightarrow 0}{\longrightarrow} 0, \,\,\, \text{in}\,\, W^{-1,p}(\Omega)$
\end{center}
and
\[
\|u_\varepsilon^-\|_{L^{p^*}(\Omega)}\rightarrow 0 , \,\quad \text{as}\,\,\, \varepsilon\rightarrow 0\,.
\]
 Then, passing to a subsequence if  it is necessary, there exist $\lambda_i\in\mathbb{R}_+$, $x_i\in\mathbb{R}^N$, sequences
$\{y_\varepsilon^i\} \subset \Omega$  and $\{\delta_\varepsilon^i\} \subset \R_+,$ with $i=1,\ldots,k$,  satisfying
$$
\frac{1}{\delta_\varepsilon^i}\, \textrm{dist} \, (y_\varepsilon^i,\partial\Omega)\rightarrow \infty , \,\quad \text{as}\,\,\, \varepsilon\rightarrow 0,
$$
and
\begin{equation}\label{trisprojection}
\|u_\varepsilon(\cdot)-v_0-\overset{k}{\underset{i=1}{\sum}}
(\delta_\varepsilon ^i)^{(p-N)/p}\mathcal{U}_{\lambda_i,x_i} ((\cdot-y_\varepsilon^i)/\delta_\varepsilon^i)\|_{\mathcal {D}^{1,p}(\Omega)}\rightarrow 0, \quad \text{as}\,\,\,\varepsilon\rightarrow 0\,,
\end{equation}
for some $v_0\in W^{1,p}_0(\Omega)$, nonnegative solution to $-\Delta_p w =w^{p^*-1}$ in $\Omega$.\\
Furthermore,
\begin{equation}\nonumber
\|u_\varepsilon\|_{\mathcal {D}^{1,p}(\Omega)}^p\rightarrow \|v_0\|_{\mathcal {D}^{1,p}(\Omega)}^p
+
\overset{k}{\underset{i=1}{\sum}}
\|\mathcal{U}_{\lambda_i,x_i} \|_{\mathcal {D}^{1,p}(\Omega)}^p, \quad \text{as}\,\,\,\varepsilon\rightarrow 0\,,
\end{equation}
and
\begin{equation}\label{trisssprojectionenetimateenery}
J(v_0)+
\overset{k}{\underset{i=1}{\sum}}
J_\infty (\mathcal{U}_{\lambda_i,x_i} )= c\,.
\end{equation}
\end{theorem}

In a forthcoming paper (\cite{mercuri}),  exploiting Theorem \ref{triscorojfmdmgmgmgbmbxbdkjgdfkghk},  it is proved
 that the critical problem  for the p-Laplacian  can admit  positive (nontrivial)
solutions in some bounded domains that are not star shaped. This extends to the quasilinear context the classical Coron's result \cite{coron}.\\

\noindent Furthermore, let us point out that, as in the semilinear case (see e.g.  \cite{bari,wei,han,rey}), our result can be applied in the  asymptotic analysis of positive solutions to  slightly subcritical problems.
More precisely, given  $\e>0$ small, let us consider the following problem, where $\Omega $   is a bounded  smooth domain
\begin{equation}\label{aluigisegustaelcordero}
\begin{cases}
-\Delta_p u_{\varepsilon}=u_{\varepsilon}^{p^*-1-\varepsilon}\hspace{5pt}&\hbox{in}\;\;\Omega,\\
u_{\varepsilon}=0, & \text{ on } \partial\Omega\,,
\end{cases}
\end{equation}

An analysis of \emph{Mountain Pass} solutions to \eqref{aluigisegustaelcordero} has been  developed in
\cite{peral2} (see also \cite{peral3}). More precisely, exploiting the \emph{concentration compactness principle}
of P.L. Lions (see \cite{L1,L2}), it is proved in \cite{peral2} that the sequence of mountain pass solutions
 converges in the sense of measures to a unique Dirac $\delta$  with mass $S^{\frac{N}{p}}$, being
$S$ the best Sobolev's constant (see \eqref{SSS}). Namely, the limit energy level is exactly $c=S^{\frac{N}{p}}/N$ in this case.\\

 \noindent On the contrary, when dealing with  higher energy level solutions, many problems remained unsolved because of the lack
 of a classification result for the critical problem in $\R^N$.\\
 In \cite[Theorem 9]{peral2} the authors also considered the second energy level   $ S^{\frac{N}{p}}/N < c<2S^{\frac{N}{p}}/N$
 proving that the following alternative holds:
 \begin{itemize}
 \item[(a)] the sequence converges to a solution of the critical problem in $\Omega$.
 \item[(b)]  the sequence converges to a Dirac mass.
 \end{itemize}
 As a consequence of Theorem \ref{triscorojfmdmgmgmgbmbxbdkjgdfkghk}, we deduce that the second alternative (b) cannot occur in the case
 \[
  S^{\frac{N}{p}}/N < c<2S^{\frac{N}{p}}/N\,.
  \]
  More precisely, we have the following:
\begin{theorem}\label{corojfmdmgmgmgbmbxbdkjgdfkghk}
Let $1<p<2$ and assume that $p^*\geqslant 2$, namely $\frac{2N}{N+2}\leqslant p<2$.
Let $\{u_\e\}$ be a family of nonnegative solutions to \eqref{aluigisegustaelcordero} such that
\begin{center}
$J(u_\varepsilon)\underset{\varepsilon\rightarrow 0}{\longrightarrow} c\quad$ and
 $\quad J'(u_\varepsilon)\underset{\varepsilon\rightarrow 0}{\longrightarrow} 0, \,\,\, \text{in}\,\, W^{-1,p}(\Omega)$
\end{center}
with $\frac{S^{\frac{N}{p}}}{N}\leqslant c<2 \frac{S^{\frac{N}{p}}}{N}$.
Then, the following alternative holds:

\begin{itemize}
\item[$(i)$] the function $v_0$ in  \eqref{trisprojection} is not trivial. In this case, \eqref{trisprojection} is fulfilled with $k=0$, namely
    \begin{equation}\nonumber
\|u_\varepsilon(\cdot)-v_0\|_{\mathcal {D}^{1,p}(\Omega)}\rightarrow 0, \quad \text{as}\,\,\,\varepsilon\rightarrow 0\,.
\end{equation}
\item[$(ii)$]
The function $v_0$ in  \eqref{trisprojection} is  trivial. This case can occur only if $c=\frac{S^{\frac{N}{p}}}{N}$ and  \eqref{trisprojection} is fulfilled with $k=1$, namely
there exist $\lambda\in\mathbb{R}_+$, $x_0\in\mathbb{R}^N$, sequences
$\{y_\varepsilon\} \subset \Omega$  and $\{\delta_\varepsilon\} \subset \R_+,$  satisfying
$$
\frac{1}{\delta_\varepsilon}\, \textrm{dist} \, (y_\varepsilon,\partial\Omega)\rightarrow \infty , \,\quad \text{as}\,\,\, \varepsilon\rightarrow 0,
$$
and
\begin{equation}\label{djfhjkdhgkghkvcgvcgvcgdcvg}
\|u_\varepsilon(\cdot)-\delta_\varepsilon^{(p-N)/p}\mathcal{U}_{\lambda,x_0} ((\cdot-y_\varepsilon)/\delta_\varepsilon)\|_{\mathcal {D}^{1,p}(\Omega)}\rightarrow 0, \quad \text{as}\,\,\,\varepsilon\rightarrow 0\,.
\end{equation}
 \end{itemize}
 As a consequence, if $S^{\frac{N}{p}}/N < c<2S^{\frac{N}{p}}/N$, then $(i)$ occurs.
\end{theorem}
\begin{remark}
For the reader's convenience we point out that, changing the parameters,  \eqref{djfhjkdhgkghkvcgvcgvcgdcvg} in Theorem \ref{corojfmdmgmgmgbmbxbdkjgdfkghk} holds with
  $\lambda=1$ and $x_0=0$.
  Namely,
$y_\varepsilon$ has to be replaced by $y_\varepsilon+\delta_\varepsilon x_0$ and $\delta_\varepsilon$ has to be replaced by $\lambda\delta_\varepsilon$.
\end{remark}
\begin{remark}
Note that in Theorem \ref{corojfmdmgmgmgbmbxbdkjgdfkghk} the case $(i)$ is not possible if the domain is star shaped. In fact in this case, by a Pohozaev type identity (see \cite{lucdin,deg,guedda}), it follows that $v_0$ has to be trivial.\par
 If instead the domain is not star shaped, then $v_0$ might be not trivial. \\
This is exactly what happens in the Coron type problem studied in \cite{mercuri}, which exploits our result. In this case the energy level is, in fact, strictly grater than $S^{\frac{N}{p}}/N$.
\end{remark}

\noindent The paper is organized as follows: in Section \ref{skudgskughgkvccb} we study the symmetry of the solutions in $\R^N$ and we prove Theorem \ref{radialtheorem}. In Section \ref{sectioncorta} we prove
 Theorem \ref{triscorojfmdmgmgmgbmbxbdkjgdfkghk} and Theorem \ref{corojfmdmgmgmgbmbxbdkjgdfkghk}.

\section{Symmetry of solutions}\label{skudgskughgkvccb}
\begin{definition}\label{sol} We say that $u$ is a nontrivial nonnegative weak solution to \eqref{eq:p} if $u\in \mathcal {D}^{1,p}(\R^N)$, $u \not \equiv 0$,
 $u \geq 0$ in $\R^N$, and
$$\int_{\mathbb{R}^N}|\nabla u|^{p-2}\langle\nabla u,\nabla \phi\rangle\, dx=\int_{\mathbb{R}^N}u^{p^*-1}\phi \,dx,\quad \forall \phi \in C^{\infty}_c(\mathbb{R}^N).$$
\end{definition}
As already remarked a nontrivial nonnegative solution is in fact  bounded, positive in $\R^N$ and belongs to the class $C^{1,\alpha}_{loc}(\mathbb{R}^N)$.\\
In the sequel we use the following standard estimate, whose proof can be found e.g. in \cite{lucio}.
 \begin{lemma}
$ \forall \, p>1 $   there exist positive constants $C_1, C_2$,  depending on $p$, such that  $\; \forall \, \eta, \eta' \in  \mathbb{R}^{N}$ with $|\eta|+|\eta'|>0$
\begin{eqnarray}\label{eq:inequalities}
[|\eta|^{p-2}\eta-|\eta'|^{p-2}\eta'][\eta- \eta'] \geq C_1 (|\eta|+|\eta'|)^{p-2}|\eta-\eta'|^2, \\ \nonumber
\\ \nonumber
||\eta|^{p-2}\eta-|\eta'|^{p-2}\eta'| \leq C_2|\eta -\eta'|^{p-1}, \qquad \text{if }\quad 1<p \leq2.
\end{eqnarray}
 \end{lemma}
Next we recall some known results about solutions of
differential inequalities involving the $p$ - Laplace operator.  We
begin
with a strong comparison principle whose proof can be found in
\cite{lucio} (see also \cite{PSB}).
\begin{theorem}[Strong comparison principle]
Suppose that $\Omega$ is a domain in $\mathbb{R} ^N$ and $u, v \in C^1
(\Omega)$
weakly solve
\begin{eqnarray}\nonumber
\begin{array}{llll}
\left \{ \begin{array}{llll}
- \Delta_p u &\geq& f(u) & {\rm in} \; \Omega ,\\
- \Delta_p v &\leq& f(v) & {\rm in} \; \Omega ,
\end{array}\right.\end{array}
\end{eqnarray}
with $f : \mathbb{R} \rightarrow \mathbb{R} $ locally Lipschitz
continuous .\\
 Suppose that $1 < p < \infty$ and define
$Z^u_v =
\{x \in
\Omega : \nabla u (x) = \nabla v (x) = 0 \}$.  If $ u \geq v \; {\rm in}\; \Omega$
and
there exists $x_0 \, \in \, \Omega \setminus Z^u_v$ with $u(x_0) =
v(x_0)$,
then $u \equiv v$ in the
connected component of $\Omega \setminus Z^u_v \;$ containing $x_0$.
\end{theorem}

Finally, we recall a version of the strong maximum principle
and the
Hopf's lemma for the $p$-Laplacian which is a particular case of a
result
proved in \cite{V}  (see also \cite{PSB} and
\cite{lucio}).

\begin{theorem}[Strong maximum principle
and Hopf's Lemma]  Let $\Omega$ be a domain in $\mathbb{R}^N$
and suppose that $u \in C^1 (\Omega), u \geq 0 $ in $ \Omega$ weakly
solves
\medskip
\begin{equation}\nonumber
- \Delta_p u + c u^q = g \geq 0 \quad {\rm in}\; \Omega
\end{equation}
\noindent  with $1 < p <\infty, \; q \geq p - 1,\; c \geq 0$ and $g
\, \in \,
L_{loc}^\infty (\Omega)$.  If $u \not \equiv 0$, then $u > 0$ in
$\Omega$.  Moreover, for any point $x_0 \in \partial \Omega$ where the
interior sphere condition is satisfied and
 $u \in C^1 (\Omega \cup \{x_0\})$ and $u(x_0) = 0$,  we have that
$\frac{\partial u}{\partial s} > 0$ for any inward directional
derivative (this means that if $y$ approaches $x_0$ in a ball
$B \subset \Omega$ that has $x_0$ on its boundary then
$\lim _{y \to x_0} \frac{u (y) - u (x_0)}{|y - x_0|} > 0.$)

\end{theorem}

\medskip

To state the next results we need some notations. \\
If $\nu$ is a direction in $\mathbb{R}^N$, i.e. $\nu  \in \mathbb{R}^N$ and  $|\nu|=1$, and $\lambda$ is  a real number  we set
\begin{equation}\nonumber
T_\lambda^\nu=\{x\in \mathbb{R}^N:x\cdot\nu=\lambda\}.
\end{equation}
\begin{equation}\nonumber
\Sigma_\lambda^\nu=\{x\in \mathbb{R}^N:x\cdot \nu <\lambda\},
\end{equation}
\begin{equation}\nonumber
x_\lambda^\nu=R_\lambda^\nu(x)=x+2(\lambda -x\cdot\nu)\nu,
\end{equation}
(i.e. $R_\lambda^\nu $  is the reflection trough the hyperplane $T_\lambda^\nu$),
\begin{equation}\nonumber
u_{\lambda }^\nu(x)=u(x_\lambda^\nu) \,.
\end{equation}
We denote by $Z_u$ the critical set of $u$ defined by
$Z_u=\{x\in \R^N: \nabla u(x)=0\} $ and we put
\begin{equation}\nonumber
Z_{\lambda }^\nu=\{x\in \Sigma_\lambda^\nu: \nabla u(x)= \nabla u_{\lambda}^\nu(x)=0\} \subseteq Z_u.
\end{equation}
Finally we define
\begin{equation}\nonumber
\Lambda(\nu)=\{\lambda\in \mathbb{R}\,:\, u\leq u_{\mu}^\nu\,\, \text{in}\,\, \Sigma_\mu^\nu,\,\, \forall \mu\leq\lambda \}.
\end{equation}
 and if $\Lambda (\nu)\neq \emptyset$ we set
\begin{equation}\label{eq:suplambda}
 \lambda _0 (\nu) =\sup \Lambda (\nu) .
\end{equation}
We will refer to $ \Sigma_{\lambda _0}^\nu$ as the maximal cap.
As a first step toward the proof of Theorem \ref{radialtheorem} we state and prove the following partial symmetry result.

\begin{theorem}[Partial symmetry]\label{Partial symmetry}  Let $1<p<2$ with $p^*\geqslant 2$, namely $\frac{2N}{N+2}\leqslant p<2$ and let $u$  be a (positive) solution to \eqref{eq:p} in the sense of Definition~\ref{sol}. \\
Then, for every direction $\nu$ we have that $\Lambda (\nu)\neq \emptyset$,  $ \lambda _0(\nu)$ in \eqref{eq:suplambda} is (well defined and) finite,  and one of the following conclusions  occurs
\begin{itemize}
\item [$(i)$] either $u \equiv u_{\lambda _0(\nu)}^\nu$ in $\Sigma_{ \lambda _0(\nu)}^\nu$;
\item[$(ii)$] or $u$ has a local symmetry region, namely
$u \not \equiv u_{\lambda _0(\nu)}^\nu$ in $\Sigma_{ \lambda _0(\nu)}^\nu$ but
$ u\equiv u_{\lambda _0(\nu)}^\nu\,\, \text{in}\,\, \mathcal{C}^\nu,$
for some connected component $\mathcal{C}^\nu$ of
$ \Sigma_{\lambda _0(\nu)}^\nu \setminus Z_{\lambda _0 (\nu) }^\nu$.
\end{itemize}
In the latter case, for any  connected component $\mathcal{C}^\nu$  of local symmetry, we also have
\begin{equation} \label{gradientediversodazeronelle componenti}
\nabla u(x) \neq 0 \quad \forall \; x \in C^\nu \; {\rm and} \; |\nabla u(x)| =|\nabla u_{\lambda _0(\nu)}^\nu |=0 \quad
\forall
\; x \in \partial C^\nu \setminus T^\nu_{\lambda_0(\nu)} \, .
\end{equation}
Moreover,
\begin{equation}\label{disug.strettaprimadilambda0}
u < u^\nu_\lambda \quad {\rm in} \quad \Sigma^\nu_\lambda \, \setminus \,
Z^\nu_\lambda ,\quad \forall \quad \lambda < \lambda_0(\nu),
\end{equation}
\begin{equation}\label{derivataparzialepositiva}
\frac{\partial u}{\partial \nu} (x) > 0 \quad \forall \quad x \in
\Sigma^\nu_{\lambda_0(\nu)} \, \setminus  \, Z \, .
\end{equation}

\end{theorem}

\begin{proof}
In what follows, without loss of generality  (being the $-\Delta_p(\cdot)$ operator invariant with respect  to rotations),  we will suppose, for simplicity of notations, that  $\nu$ coincides with  the $e_1=(1,0,\cdots,0)$ direction and  we omit the superscript in previous definitions, so that
$ \Sigma_{\lambda}=\Big\{x\in \mathbb{R}^N : x_1< \lambda\Big\}$, $ T_{\lambda}=\Big\{x\in \mathbb{R}^N : x_1= \lambda\Big\}$,
$x_\lambda=R_{\lambda }(x)=(2\lambda-x_1,x_2,\cdots,x_N)$,
$u_{\lambda}(x)=u(x_\lambda)$, $ Z_{\lambda }=\{x\in\Sigma_\lambda: \nabla u(x)= \nabla u_{\lambda}(x)=0\} \subseteq Z_u$,
 $
\Lambda :=\Lambda (e_1) =\{\lambda\in \mathbb{R}\,:\, u\leq u_{\mu}\,\, \text{in}\,\, \Sigma_\mu,\,\, \forall \mu\leq\lambda \}$. \par
To prove the proposition, we shall use the Moving Plane Method  in the $e_1$-direction.\\

\noindent \textbf{STEP 1}
{ \it Here we show that $\Lambda \neq \emptyset $, so that we can define
$
 \lambda _0 =\sup\Lambda
$, and that $
 \lambda _0 $ is finite.} \\

 \noindent We need to consider two different cases: the case $p^*>2$ (namely $p>\frac{2N}{N+2}$), and the case $p^*=2$ (namely $p=\frac{2N}{N+2}$).
\noindent \emph{Let us first deal with the case} $p^*>2$.\\
 \noindent By the summability assumptions on the solution it is
possible to take the function $ (u-v)^+$, $v=u_{\lambda}$, as a
test function in the equations for $u$ and in the equation for $v=u_{\lambda} $ in $\Sigma _{\lambda}$.
More precisely, since $u$,$v$ belong to
$ \mathcal{D}^{1,p}(\rn) $ and they coincide on the hyperplane $T_{\lambda}$, there exists a sequence
 $ \varphi _j $
of functions in $C_c^{\infty}(\Sigma _{\lambda})$ such that
$\varphi _j \to (v-u)^+$ in $L^{p^*}(\Sigma _{\lambda})$
 and $\nabla  \varphi _j \to \nabla (v-u)^+  $
in $L^p (\Sigma _{\lambda})$. Moreover, passing to a subsequence and substituting if
necessary $ \varphi_j$ with $  \varphi_j ^+$, we can suppose that
$ 0 \leq  \varphi _j \to (u-v)^+$, $\nabla  \varphi _j \to \nabla (u-v)^+  $
a.e. in $\Sigma _{\lambda}$, and that there exist  functions $ \psi _0 \in L^{p^\ast }$,
$  \psi _1 \in L^p$, such that
$ | \varphi _j | \leq \psi _0 $, $ |\nabla  \varphi _j| \leq \psi _1 $ a.e. in $\Sigma _{\lambda}$. \par
Taking the functions $\varphi _j $
  as test functions in the equation for   $u$  we get
$$
\int_{\Sigma_\lambda}\langle |\nabla u|^{p-2}\nabla u,\nabla \varphi _j \rangle dx=
\int_{\Sigma_\lambda} u^{p^*-1}  \, \varphi _j dx
$$
If we can pass to the limit for $j \to \infty$,   we obtain

\begin{equation}\label{eq:secondaaaa}
 \int_{\Sigma_\lambda}|\nabla u|^{p-2}\langle \nabla u,\nabla (u-u_{\lambda})^+\rangle dx=\int_{\Sigma_\lambda}u^{p^*-1}(u-u_{\lambda})^+ dx
\end{equation}

So it is enough to justify the passage to the limit, which follows
easily from the
 dominated convergence theorem.
In fact, we have that
$ | u^{p^*-1} \,  \varphi _j | \leq
u^{p^*-1} \psi _0 \in L^1(\Sigma_\lambda)
 $
 since
$ \psi _0 \in  L^{p^*}(\Sigma_\lambda)$,  \, $ u^{p^*-1}   \in L^{\frac {p^*}{p^*-1}}(\Sigma_\lambda) $.
Analogously we have that
$ | |\nabla u|^{p-2}\langle \nabla u,\nabla \varphi  \rangle | \leq
 |\nabla u|^{p-1}  \psi _1 \in L^1(\Sigma_\lambda)
$
 since
$ \psi _1 \in  L^{p}(\Sigma_\lambda)$,  \, $ |\nabla u|^{p-1}   \in L^{\frac {p}{p-1}}(\Sigma_\lambda) $.

Since  $u_{\lambda}$ satisfies the same equation, analogously we get
\begin{equation}\label{eq:secondaaaaA}
\int_{\Sigma_\lambda}|\nabla u_{\lambda}|^{p-2}\langle \nabla u_{\lambda},\nabla (u-u_{\lambda})^+\rangle dx=\int_{\Sigma_\lambda}{u_{\lambda}}^{p^*-1}(u-u_{\lambda})^+dx
\end{equation}
Subtracting equation \eqref{eq:secondaaaaA} from \eqref{eq:secondaaaa} we get
\begin{equation}\nonumber
\int_{\Sigma_\lambda}\langle |\nabla u|^{p-2}\nabla u- |\nabla u_{\lambda}|^{p-2}\nabla u_{\lambda},\nabla (u-u_{\lambda})^+\rangle dx=\int_{\Sigma_\lambda}(u^{p^*-1}-{u_{\lambda}}^{p^*-1}) (u-u_{\lambda})^+ dx
\end{equation}
and, by \eqref{eq:inequalities}, we have
\begin{equation}\label{eq:madriddddd}
\int_{\Sigma_\lambda}(|\nabla u|+|\nabla u_{\lambda}|)^{p-2}|\nabla(u-u_{\lambda})^+|^2dx\leq \frac{1}{C_1}\int_{\Sigma_\lambda}(u^{p^*-1}-u_{\lambda}^{p^*-1}) (u-u_{\lambda})^+ dx\,.
\end{equation}

Since we are assuming now $p^*>2$,  using Lagrange Theorem,  H\"older's and Sobolev's inequalities on the right hand side of \eqref{eq:madriddddd}, we obtain
\begin{eqnarray}\label{eq:batiado}\\\nonumber
&&\int_{\Sigma_\lambda}(u^{p^*-1}-u_{\lambda}^{p^*-1}) (u-u_{\lambda})^+ dx\leq (p^*-1)\int_{\Sigma_\lambda}u^{p^*-2} [(u-u_{\lambda})^+]^2 dx
 \\\nonumber &\leq&(p^*-1)\left(\int_{\Sigma_\lambda}u^{p^*}dx\right)^\frac{p^*-2}{p^*}\left(\int_{\Sigma_\lambda}[(u-u_{\lambda})^+]^{p^*}dx\right)^\frac{2}{p^*}\\\nonumber
&\leq& \frac{(p^*-1)}{S^{2/p}}\left(\int_{\Sigma_\lambda}u^{p^*} dx\right)^\frac{p^*-2}{p^*}\left(\int_{\Sigma_\lambda}|\nabla (u-u_{\lambda})^+|^pdx\right)^\frac{2}{p},
\end{eqnarray}
being  $S$ the Sobolev's constant (see\eqref{SSS}).
Let us note now that by H\"older's  inequality we have
\begin{eqnarray}\label{eq:cazzominchia}\\\nonumber
&&\int_{\Sigma_\lambda}|\nabla (u-u_{\lambda})^+|^pdx=\int_{\Sigma_\lambda}
 \left [ (|\nabla u|+|\nabla u_{\lambda}|)^\frac{p(2-p)}{2} \right ]  \left [(|\nabla u|+|\nabla u_{\lambda}|)^\frac{p(p-2)}{2}    |\nabla (u-u_{\lambda})^+|^p \right ] dx\\\nonumber
&\leq&
\left(\int_{\Sigma_\lambda}(|\nabla u|+|\nabla u_{\lambda}|)^pdx\right)^{\frac{2-p}{2}}\left(\int_{\Sigma_\lambda}{(|\nabla u|+|\nabla u_{\lambda}|)^{p-2}}|\nabla (u-u_{\lambda})^+|^2dx\right)^{\frac p2}\\\nonumber
&\leq&C_2\left(\int_{\Sigma_\lambda}{(|\nabla u|+|\nabla u_{\lambda}|)^{p-2}}|\nabla (u-u_{\lambda})^+|^2dx\right)^{\frac p2},
\end{eqnarray}
where we used that $\nabla u$ and $\nabla u_{\lambda}$ are both in $L^p(\mathbb{R}^N)$.\\

Let us remark, for later use,  that here we absorbed the term
$ \left(\int_{\Sigma_\lambda}(|\nabla u|+|\nabla u_{\lambda}|)^p dx \right)^{\frac{2-p}{2}} $ in the constant $C_2$, but for the limit case we will need to estimate this term. \\

Combining equations \eqref{eq:madriddddd}, \eqref{eq:batiado} and  \eqref{eq:cazzominchia} it follows that
\begin{align}\label{eq:tradddddimm}
&\int_{\Sigma_\lambda}{(|\nabla u|+|\nabla u_{\lambda}|)^{p-2}}|\nabla (u-u_{\lambda})^+|^2dx\\
\nonumber &\leq \frac{C_2^{\frac{2}{p}}(p^*-1)}{C_1\,S^{\frac{2}{p}}}\left(\int_{\Sigma_\lambda}u^{p^*} dx\right)^\frac{p^*-2}{p^*}\int_{\Sigma_\lambda}{(|\nabla u|+|\nabla u_{\lambda}|)^{p-2}}|\nabla (u-u_{\lambda})^+|^2dx.
\end{align}
Since $u\in L^{p^*}(\mathbb{R}^N)$ there exists $R_0 >0 $ such that
\begin{equation} \label{infcomunealledirezioni}
 \frac{C_2^{\frac{2}{p}}(p^*-1)}{C_1\,S^{\frac{2}{p}}}\left( \int_{\mathbb{R}^N\setminus B(0,R_0)}u^{p^*}dx \right)^\frac{p^*-2}{p^*}<\frac 12
 \end{equation}
  so that if $\lambda < - R_0 $ then
$\Sigma_\lambda \subset \mathbb{R}^N\setminus B(0,R_0) $ and
 we  deduce from \eqref{eq:tradddddimm} and \eqref{infcomunealledirezioni}  that $(u-u_{\lambda })^+\equiv0$. \\
So we proved that  $\Lambda \neq \emptyset $, in fact $(- \infty, - R_0) \subset \Lambda $.\par
We remark that for simplicity of notations we assumed that we are dealing with the $x_1$-direction, but $R_0$ is independent of the direction we are a looking at, namely $(- \infty, - R_0) \subset \Lambda (\nu )$ for every direction $\nu \in S^{N-1}$.\\


\noindent \emph{Let us now deal with the limit case} $p^*=2$ (\emph{namely} $p=\frac{2N}{N+2}$).
We can repeat verbatim the previous argument that we used to prove \eqref{eq:tradddddimm} and since $p^* -1 =1$ we get
\begin{align}\nonumber
&\int_{\Sigma_\lambda}{(|\nabla u|+|\nabla u_{\lambda}|)^{p-2}}|\nabla (u-u_{\lambda})^+|^2dx\\
\nonumber &\leq \frac{1}{C_1\,S^{\frac{2}{p}}}
\left(\int_{\Sigma_\lambda} (|\nabla u|+|\nabla u_{\lambda}|)^{p}\,\chi_{\{u\geqslant u_\lambda\}}\,dx\right)^{\frac{2-p}{p}}
\int_{\Sigma_\lambda}{(|\nabla u|+|\nabla u_{\lambda}|)^{p-2}}|\nabla (u-u_{\lambda})^+|^2dx
\end{align}
where we emphasized the factor  $\chi_{\{u\geqslant u_\lambda\}}$ (where $\chi $ is the characteristic function of a set),  and the term
$ \left(\int_{\Sigma_\lambda}(|\nabla u|+|\nabla u_{\lambda}|)^p dx \right)^{\frac{2-p}{p}} $ will be estimated, and not absorbed in the constant
$C_2^{\frac{2}{p}}$. \\
To prove that  $(u-u_{\lambda })^+\equiv0$ for $\lambda$ negative with $|\lambda|$ large, it is sufficient to show that
\begin{equation}\label{claimcasolimite}
\frac{1}{C_1\,S^{\frac{2}{p}}}
\left(\int_{\Sigma_\lambda} (|\nabla u|+|\nabla u_{\lambda}|)^{p}\,\chi_{\{u\geqslant u_\lambda\}}\,dx\right)^{\frac{2-p}{p}}<1\,.
\end{equation}
By simple computations we see that for some  $c_p>0$ depending only on $p$
\begin{equation}\nonumber
\begin{split}
&\frac{1}{C_1\,S^{\frac{2}{p}}}
\left(\int_{\Sigma_\lambda} (|\nabla u|+|\nabla u_{\lambda}|)^{p}\,\chi_{\{u\geqslant u_\lambda\}}\,dx\right)^{\frac{2-p}{p}}\\
&\leqslant \frac{c_p}{C_1\,S^{\frac{2}{p}}}
\left[
\left(\int_{\Sigma_\lambda} |\nabla u|^{p}\,\chi_{\{u\geqslant u_\lambda\}}\,dx\right)^{\frac{2-p}{p}}
+ \left(\int_{\Sigma_\lambda} |\nabla u_\lambda|^{p}\,\chi_{\{u\geqslant u_\lambda\}}\,dx\right)^{\frac{2-p}{p}}
\right]\,.
\end{split}
\end{equation}
Therefore, to prove \eqref{claimcasolimite}, it is sufficient to show  that, for $\lambda$ negative and $|\lambda|$ large both the terms in the last line are arbitrarily small, e.g.
\begin{equation}\nonumber
\int_{\Sigma_\lambda} |\nabla u|^{p}\,\chi_{\{u\geqslant u_\lambda\}}\,dx  \quad  , \quad
\int_{\Sigma_\lambda} |\nabla u_\lambda|^{p}\,\chi_{\{u\geqslant u_\lambda\}}\,dx \;
\leq \, \left(\frac{C_1\,S^{\frac{2}{p}}}{4c_p}\right)^{\frac{p}{2-p}}\,=:\,\hat\tau(p)\,.
\end{equation}

Since $u\in \mathcal {D}^{1,p}(\R^N)$, then there exists $R_0 ' > 0 $  such that
\begin{equation} \nonumber \int_{\R^N \setminus B(0, R_0')} |\nabla u|^{p}\, dx  < \hat\tau(p)\,.
\end{equation}
So,  if $\lambda < - R_0 '$ we have that
$\Sigma_\lambda \subset  \R^N \setminus B(0, R_0') $ and
$ \int_{\Sigma_\lambda} |\nabla u|^{p}\,\chi_{\{u\geqslant u_\lambda\}}\,dx  < \hat\tau(p)$. \\

To prove \eqref{claimcasolimite}, it is now sufficient to show  that, for $\lambda$ negative and $|\lambda|$ large, we have also for the other term
\begin{equation}\nonumber
\int_{\Sigma_\lambda} |\nabla u_\lambda|^{p}\,\chi_{\{u\geqslant u_\lambda\}}\,dx
<\left(\frac{C_1\,S^{\frac{2}{p}}}{4c_p}\right)^{\frac{p}{2-p}}\,=:\,\hat\tau(p)\,.
\end{equation}
Since $u \in \mathcal {D}^{1,p}(\R^N) \cap C^{1,\alpha}_{loc}(\mathbb{R}^N)$ and $u >0$ in $\R^N$,  we have that
$ \lim _{\delta \to 0} |\nabla u|^{p}\,\chi_{\{u\leqslant \delta\}} =0 $, so that by dominated convergence
there exists $\hat\delta>0$ small, depending only on $u$ and not depending on $\lambda$, such that
\begin{equation}\label{jkgkhgkhgdkgdkvcvcvc}
\int_{\R^N} |\nabla u_\lambda|^{p}\,\chi_{\{u_\lambda\leqslant \hat\delta\}}\,dx
= \int_{\R^N} |\nabla u|^{p}\,\chi_{\{u\leqslant \hat\delta\}}\,dx
\leqslant\frac{\hat\tau(p)}{4}\,.
\end{equation}
Furthermore, since $u_\lambda\in \mathcal {D}^{1,p}(\R^N)$, there exists $\bar\delta$, depending only on $u$ and not depending on $\lambda$, such that, since
$\mathcal L(\mathcal E)= \mathcal L(\mathcal R_\lambda (\mathcal E))$ (here $\mathcal L(\cdot)$ stands for the Lebesgue measure) we have that
  \begin{equation}\label{misuraE}
  \text { if } \mathcal L(\mathcal E)  \leqslant \bar \delta \; \text{ then, } \;
\int_{\mathcal E} |\nabla u_\lambda|^{p}\,dx\,=\,\int_{R_\lambda (\mathcal E)} |\nabla u|^{p}\,dx
\leqslant \frac{\hat\tau(p)}{4} \; . \quad
\end{equation}

Note now that, for $\hat\delta$ fixed as in \eqref{jkgkhgkhgdkgdkvcvcvc} and $\bar\delta$ as in \eqref{misuraE}, there exists $R_0 '' > 0 $  such that
\begin{equation} \nonumber
\int_{\R^N \setminus B(0, R_0'')}\,u^{p^*}\,dx\leqslant \bar\delta\, \hat\delta^{p^*}
\end{equation}
So,  if $\lambda < - R_0 ''$ we have that
$\Sigma_\lambda \subset  \R^N \setminus B(0, R_0'') $ and
 \begin{equation}\nonumber
\mathcal{L} \left(\{u\geqslant u_\lambda\}\cap\{u_\lambda \geqslant \hat\delta\}\cap\Sigma_\lambda\right)
\leqslant \mathcal{L} \left(\{u \geqslant \hat\delta\}\cap\Sigma_\lambda\right)\leqslant \frac{1}{\hat\delta^{p^*}}
\int_{\Sigma_\lambda}\,u^{p^*}\,dx\leqslant \bar\delta\,.
\end{equation}
Using \eqref{misuraE} we have \ \
$ \int_{\Sigma_\lambda} |\nabla u_\lambda|^{p}\,\chi_{\{u\geqslant u_\lambda\}}\chi_{\{ u_\lambda\geqslant \hat\delta\}}\,dx
\leq \frac{\hat\tau(p)}{4} $ \ \
and exploiting also \eqref{jkgkhgkhgdkgdkvcvcvc}
\begin{equation}\nonumber
\begin{split}
&\int_{\Sigma_\lambda} |\nabla u_\lambda|^{p}\,\chi_{\{u\geqslant u_\lambda\}}\,dx\\
&=\int_{\Sigma_\lambda} |\nabla u_\lambda|^{p}\,\chi_{\{u\geqslant u_\lambda\}}\chi_{\{ u_\lambda\geqslant \hat\delta\}}\,dx
+\int_{\Sigma_\lambda} |\nabla u_\lambda|^{p}\,\chi_{\{u\geqslant u_\lambda\}}\,\chi_{\{ u_\lambda< \hat\delta\}}\,dx\\
&\leq \frac{\hat\tau(p)}{2}\,,
\end{split}
\end{equation}
and \eqref{claimcasolimite} is proved, so that  also in this case we have $(u-u_{\lambda })^+\equiv0$ for $\lambda$ negative with $|\lambda|$ large.\\

 It is easy to see  that also in the limit case the values that we have to consider for $\lambda$ do not depend on the direction that we consider.
 In fact  $R_0'  $,  and $\hat\delta$ and $\bar\delta$,  which next determine  $  R_0 ''$, depend only on $u$, and do not depend on the direction under consideration, and   taking $R_0 = \max (R_0' , R_0 '')  $, we have that   $(- \infty, - R_0) \subset \Lambda (\nu )$ for every direction $\nu \in S^{N-1}$.\\

Having  proved that  $\Lambda \neq \emptyset $  we can define
$
 \lambda _0 =\sup\Lambda
$ and it easy to see that
$$ \lambda _0 <+\infty
$$
 since $u$ is positive and  $u\in L^{p^*}(\mathbb{R}^N)$: if we had   $ \lambda _0 =+\infty $ then, $u$ would be  positive and nondecreasing in the $x_1$-direction in $\R^N$.\\

\noindent \textbf {STEP 2}
{ \it Here we show that in the maximal cap $ \Sigma_{\lambda _0}$ either i) or ii) occur.} \\

\noindent \emph{Let us first deal with the case} $p^*>2$.
Let  $\mathcal{C}$ be a connected component of $ \Sigma_{\lambda _0}\setminus Z_{\lambda _0}$.
By the Strong Comparison Principle it follows that $u<u_{\lambda}$ or $u\equiv u_{\lambda}$ in all $\mathcal{C}$.

If $u \equiv u_{\lambda _0} $ in $\Sigma_{ \lambda _0}$ then  $(i)$ holds.  \\
If instead
 $u \not \equiv u_{\lambda _0}$ in $\Sigma_{ \lambda _0}$ but  $u  \equiv u_{\lambda _0}$  in at least one connected component $\mathcal{C}$ then $(ii)$ holds. \par

Therefore, we have to show that it is not possible that $u< u_{\lambda _0}$ in  \emph{every} connected component $\mathcal{C}$ of $ \Sigma_{\lambda _0}\setminus Z_{\lambda _0}$.\par

Suppose by contradiction that $u<u_{\lambda _0} $ in $ \Sigma_{\lambda _0}\setminus Z_{\lambda _0}$ and let
 $\eta>0$ be arbitrarily small (to be fixed later). \\
 By the results in \cite{DS1}, \cite{Sci1} and \cite{Sci2}, we know that    $\mathcal L(Z_u)=0,$ where $\mathcal L(A)$ is the Lebesgue measure of a set $A$. \\
This and the fact that   $u\in L^{p^*}(\mathbb{R}^N)$ allow us to fix a compact
 $K=K _{\eta } \subset \Sigma_{\lambda _0} \setminus Z_u \subset  \Sigma_{\lambda _0} \setminus Z_{\lambda _0} $ such that
\begin{equation}\label{compattolambda0}
\int_{  \Sigma_{\lambda _0}   \setminus K  }u^{p^*}dx=\int_{ ( \Sigma_{\lambda _0} \setminus Z_u )   \setminus K  }u^{p^*}dx\leq \frac {\eta }2\,.
  \end{equation}
  By the continuity of the integrals with respect to $\lambda $,
      we can fix  $\bar{\varepsilon} >0 $,  such that   we have for $K = K_{\eta}$
\begin{align}\label{bolita}
K  \subset \Sigma_{\lambda } \setminus Z_u \, , \quad
 \int_{\Sigma_{\lambda }  \setminus K  }u^{p^*}dx\leq  \eta \quad \quad  \forall \lambda  \in ( \lambda _0 -\bar  \varepsilon , \lambda _0 +\bar{\varepsilon})
\end{align}
and,   since  $u_{{\lambda _0}}-u \geq \alpha >0 $ in $K=K_{\eta }$, by continuity we have that $u_{{\lambda }}-u$ is positive in $K$  if $\lambda $ is close to
$\lambda  _0$, so that, taking $\bar \varepsilon $ smaller if necessary we can suppose that
\begin{equation}\label{bolita2} u<u_{{\lambda }}\quad\hbox{in}\quad K \quad \quad  \forall \lambda  \in ( \lambda _0 -\bar  \varepsilon , \lambda _0 +\bar{\varepsilon})\,.
\end{equation}
For such values of $\lambda $ we consider  $ (u-u_{\lambda })^+$ as a test function in   \eqref{eq:p}
 and, since $(u- u_{\lambda})^+ \equiv 0 $ in $K= K _{\eta }$, arguing exactly as in Step 1, we obtain
\begin{align*}
&\int_{\Sigma_{\lambda}\setminus K} (|\nabla u|+|\nabla u_{\lambda}|)^{p-2}|\nabla (u-u_{\lambda })^+|^2dx \\
&\leq C\left(\int_{\Sigma_{\lambda }\setminus K } u^{p^*} dx\right)^\frac{p^*-2}{p^*}\int_{\Sigma_{\lambda }\setminus K }(|\nabla u|+|\nabla u_{\lambda}|)^{p-2}|\nabla (u-u_{\lambda })^+|^2dx.
\end{align*}

 Thanks to  \eqref{bolita},  we can assume that
 $\eta>0$ is sufficiently small so that  $\forall \, \lambda  \in (\lambda _0 -\bar  \varepsilon , \lambda _0 +\bar{\varepsilon}) $
\begin{align*}
C\left(\int_{\Sigma_{\lambda }\setminus K }u^{p^*}dx\right)^{\frac{p*-2}{p*}}<1.
\end{align*}

Therefore, for such values of $\lambda $ we obtain that $u\leq u_{{\lambda }}$ in $\Sigma_{{\lambda }}\setminus K $, and
by \eqref{bolita2} we get actually that
$u\leq u_{\lambda }$ in $\Sigma_{\lambda}$ for any  $ \lambda  \in (\lambda _0 -\bar  \varepsilon , \lambda _0 +\bar{\varepsilon})  $. In particular, this holds for $\lambda > \lambda _0 $ and close to $\lambda _0$, and
this is a contradiction with the definition of $\lambda _0$. So we can conclude that
$u\equiv u_{{\lambda _0}}$ in at least one connected component $\mathcal{C}$ of $ \Sigma_{\lambda _0}\setminus Z_{\lambda _0}$.\par
Moreover, by symmetry $|\nabla u(x)|= |\nabla u_{\lambda _0}(x)|$ if $x \in \mathcal{C} $ so that by definition \eqref{gradientediversodazeronelle componenti} follows.\\

\noindent  \emph{Let us now deal with the case} $p^*=2$ (\emph{namely} $p=\frac{2N}{N+2}$). \\
The proof is entirely analogous to the case $p^*>2$. As in Step 1, the only change is the estimation of the term
$ \int_{  \Sigma_{\lambda }   \setminus K  } (|\nabla u|+|\nabla u_{\lambda }|)^{p}\,\chi_{\{u\geqslant u_{\lambda } \}}\,dx $ instead of the term
$ \int_{  \Sigma_{\lambda }   \setminus K  }u^{p^*}dx $, but it is not  needed  any fine estimate (as was the case in Step 1 for the limit case).\\
 More precisely,
 supposing by contradiction that $u<u_{\lambda _0} $ in $ \Sigma_{\lambda _0}\setminus Z_{\lambda _0}$ and
    given $\eta >0 $ arbitrarily small, since $(|\nabla u|+|\nabla u_{\lambda _0}|)^p  \in L^1 (\Sigma _{\lambda _0})$,
   we can choose a compact set
$K=K _{\eta } \subset \Sigma_{\lambda _0} \setminus Z_u \subset  \Sigma_{\lambda _0} \setminus Z_{\lambda _0} $ such that
\begin{equation}\nonumber
\int_{  \Sigma_{\lambda _0}   \setminus K  } (|\nabla u|+|\nabla u_{\lambda _0}|)^{p}\, dx =\int_{ ( \Sigma_{\lambda _0} \setminus Z_u )   \setminus K  }  (|\nabla u|+|\nabla u_{\lambda _0}|)^{p}\, dx  \leq \frac {\eta }2\,
  \end{equation}
  and $u_{{\lambda _0}}-u \geq \alpha >0 $ in $K=K_{\eta }$.
 By  continuity   with respect to $\lambda $   there exists $\bar \varepsilon  >0 $ such that
 \begin{align}\nonumber   \forall \lambda \in (\lambda _0 - \bar \varepsilon  , \lambda _0 + \bar \varepsilon  ) \, :  \quad
  K  \subset \Sigma_{\lambda }     \quad ,  \quad \quad
 \int_{  \Sigma_{\lambda }     \setminus K  } (|\nabla u|+|\nabla u_{\lambda}|)^{p}\,\chi_{\{u\geqslant u_{\lambda } \}}\,dx
 \leq \eta
\end{align}
\begin{equation}\nonumber u<u_{{\lambda }}\quad\hbox{in}\quad K \quad \quad  \forall \lambda  \in ( \lambda _0 -\bar  \varepsilon , \lambda _0 +\bar{\varepsilon})\,.
\end{equation}
Then,     considering for the previous values of $\lambda $  the function
$ (u-u_{\lambda })^+$ as a test function in   \eqref{eq:p},
 since $(u- u_{\lambda})^+ \equiv 0 $ in $K= K _{\eta }$, arguing exactly as in Step 1, we obtain
\begin{align*}
&\int_{\Sigma_{\lambda}\setminus K} (|\nabla u|+|\nabla u_{\lambda}|)^{p-2}|\nabla (u-u_{\lambda })^+|^2dx \\
&\leq C  \left( \int_{\Sigma_{\lambda}\setminus K} (|\nabla u|+|\nabla u_{\lambda}|)^{p}\,\chi_{\{u\geqslant u_\lambda\}}\,dx \right)^{\frac{2-p}{p}}
\int_{\Sigma_{\lambda }\setminus K }(|\nabla u|+|\nabla u_{\lambda}|)^{p-2}|\nabla (u-u_{\lambda })^+|^2dx.
\end{align*}
For a  suitable choice of $\eta $ we have that
$C  \left( \int_{\Sigma_{\lambda}\setminus K} (|\nabla u|+|\nabla u_{\lambda}|)^{p}\,\chi_{\{u\geqslant u_\lambda\}}\,dx \right)^{\frac{2-p}{p}}  <1$ and we conclude as in the previous case.\\

\noindent \textbf {STEP 3}
{ \it Here we prove \eqref{disug.strettaprimadilambda0} and \eqref{derivataparzialepositiva} }. \\

 Let us remark that by the definition of $\lambda _0 $ the function $u$ is monotone nondecreasing in the $x_1$ variable in the cap $\Sigma_{\lambda _0}$.\\
Let us suppose by contradiction that for some $\mu < \lambda _0$, $y_0 < \mu$ and $x_0=(y_0,z_0) \in  \Sigma_{\mu}\setminus Z_{\mu} $ we have that $u(x_0)= u_{\mu}(x_0)$. Then, by the Strong Comparison Principle   it follows that  $u\equiv u_{\mu}$ in the connected component  $\mathcal{C}$ of $ \Sigma_{\mu}\setminus Z_{\mu}$ to which $x_0$ belongs, so that $|\nabla u (x_0)| =|\nabla u_{\mu} (x_0)| \neq 0  $. If $\mu < \lambda \leq \lambda _0 $ and $R_{\mu} ( R_{\lambda}(x_0)) \in \mathcal {C}$  we  have $u(x_0) \leq u \left ( R_{\lambda}(x_0) \right )=
u \left ( R_{\mu} ( R_{\lambda}(x_0)) \right ) $, so that for $y<y_0$, with $y_0-y$ small we have $u(y_0,z_0) \leq  u(y,z_0) $ and by the monotonicity of $u$ in the cap  $\Sigma_{\lambda _0}$ we  obtain that $u(y_0,z_0) =  u(y,z_0) $ for $y<y_0$, with $y_0-y$ small. So for $y$ belonging to a nonvoid maximal interval $(a, y_0)$ we have that $u(y_0,z_0) =  u(y,z_0) $, and it follows easily that  $a = - \infty $:
 if $a > - \infty $ repeating the previous argument with $\mu$ substituted by $\frac {a+y_0}2 $ we obtain that $u(y_0,z_0) =  u(y,z_0) $   if $y<a$ is close to $a$. \par
 So we proved that if  for some $\mu < \lambda _0$, $y_0 < \mu$ and $x_0=(y_0,z_0) \in  \Sigma_{\mu}\setminus Z_{\mu} $ we have that $u(x_0)= u_{\mu}(x_0)$  then $u(y,z_0) \equiv c $  for $ y \in (- \infty , y_0) $. \\
 The same argument can be repeated verbatim for all points close to the starting point $(y_0,z_0)$, and this easily allows to conclude  that $u$ would
 be strictly bounded away from zero in a strip $(-\infty,b)\times B_{\R^{N-1}}(z_0,r)$,
contradicting the fact that
  $u \in L^{p^*}(\R^N)$, so that \eqref{disug.strettaprimadilambda0} follows. \par
Finally, \eqref{derivataparzialepositiva}   follows from the previous inequality  and the usual Hopf's Lemma for strictly elliptic operators, since the difference $u- u_{\lambda}$ satisfies a uniformly elliptic equations in a neighborhood of any noncritical point.
In fact, let $x=(\lambda ,z) \in
\Sigma _{\lambda _0} \setminus Z $, i.e. $\lambda <\lambda _0 $
and $\nabla u(x)\neq 0 $. In a ball $B=B_r(x)$ we have that $|\nabla u|\geq \epsilon
>0 $, so that $|\nabla u|, |\nabla u_{\lambda}|\geq \epsilon >0 $ in
$B\cap \Sigma _{\lambda} $.
This implies by standard results that $u\in C^2 (B)$ and that
the difference $u_{\lambda }-u $ satisfies a linear strictly elliptic
equation $L(u_{\lambda }-u) =0 $.
On the other hand we have, by  \eqref{disug.strettaprimadilambda0} that
$u_{\lambda }-u >0 $ in $B\cap \Sigma _{\lambda} $ while
$u(x)=u_{\lambda }(x)$ because $x$ belongs to $T_{\lambda } $.
Hence, by the usual Hopf's lemma we get
$ 0> \frac {\partial (u_{\lambda }-u )} {\partial x_1}(x)=
-2\frac {\partial u} {\partial x_1}(x)  $
i.e.  \eqref{derivataparzialepositiva} holds.
\end{proof}

\begin{remark} In the proof we assumed for simplicity of notations that the direction involved was the $x_1$-direction, but some of the results are uniform w.r.t. directions. In particular:
\begin{itemize}
\item[-]  in Step 1 we proved that there exists $R_0 >0 $ such that  $(- \infty, - R_0) \subset \Lambda (\nu )$ for \emph{every} direction $\nu \in S^{N-1}$. \par
\item[-] In Step 2 we showed (for $\nu _0 = e_1$)  that if $u<u_{\lambda _0}^{\nu _0} $ in $ \Sigma _{\lambda _0}^{\nu _0} \setminus Z_{\lambda _0}^{\nu _0}$, then we can obtain the inequality $u \leq u_{\lambda}^{\nu _0}$ in $ \Sigma_{\lambda }^{\nu _0}$ for any $\lambda $ in a neighborhood of $\lambda _0 $.
 Analogously we can move direction close to $\nu _0 $ and the proof is entirely analogous.
 Namely the following holds
 \end{itemize}
\end{remark}
\begin{lemma}\label{muoveredirezionielambda}  Let $ \nu_0 $ be a direction in $\mathbb{R}^N$.
\begin{itemize}
\item[i) ] Let  $\la_0 \in  \Lambda (\nu_0)$ i.e. $\lambda _0 \leq \lambda _0 (\nu _0)$.\par
 If $ u \; < \; u^{\nu_0}_{\la_0} \quad {\rm in} \quad
(\Sigma^{\nu_0}_{\la_0} \setminus Z^{\nu_0}_{\la_0}) \,
$ (this happens in particular if $\lambda _0 < \lambda _0(\nu _0)$ by \eqref{disug.strettaprimadilambda0}),
then, there exists
$\bar  \varepsilon >0  $ such that
 $ \quad u \leq u _{\la}^{\nu } $ in $ \Sigma _{\la}^{\nu } $ \ \
$\forall $
$ \la \in  \left ( \la_0 -\bar  \varepsilon \; , \; \la_0 +\bar  \varepsilon \right ) $, \
$ \forall $ $ \nu \in I_{\bar  \varepsilon}(\nu _0 ) $,
where $I_{\bar  \varepsilon} (\nu_0) := \left\{ \nu \, : \, |\nu| = 1 \, , \,
|\nu - \nu_0| < \bar  \varepsilon \right\}$.
\item[ii) ] Let $\lambda _0  = \lambda _0 (\nu _0)$, and let
$F_{\nu _0} = \{ x \in \Sigma^{\nu_0}_{\la_0} \setminus Z^{\nu_0}_{\la_0} : u(x)= u_{\lambda _0}^{\nu _0}(x) \}$.\par
 There exists $  \bar  \varepsilon >0  $ such that for any $ \la _1 \in  \left ( \la_0 -   \bar  \varepsilon \; , \; \la_0 +   \bar  \varepsilon \right ) $
and for any direction $\nu _1 \in I_{  \bar  \varepsilon} (\nu_0) := \left\{ \nu \, : \, |\nu| = 1 \, , \,
|\nu - \nu_0| <   \bar  \varepsilon \right\}$ the following holds: \\
 if we have \ \
$
u \; < \; u^{\nu_1}_{\la_1} \quad {\rm in} \quad
F_{\nu _0} \cap \Sigma^{\nu_1}_{\la_1}  \,
$ \ \
then there exists in turn
$  \varepsilon  _1 >0$ ($  \varepsilon  _1 <  \bar  \varepsilon - |\lambda _1 - \lambda _0|$, $  \varepsilon  _1 <  \bar  \varepsilon - |\nu _1 - \nu _0|$) such that
 $ \quad u \leq u _{\la}^{\nu } $ in $ \Sigma _{\la}^{\nu } $ \ \ \
$\forall $
$ \la \in  \left ( \la_1 -  \varepsilon _1 \; , \; \la_1 +  \varepsilon _1 \right ) $, \
$ \forall $ $ \nu \in I_{  \varepsilon _1 }(\nu _1 ) $.
\end{itemize}
\end{lemma}

\begin{proof}  The proof is analogous to that of Step 2 in Theorem \ref{Partial symmetry} and we only sketch it in the case $p^* > 2$ (in the limit case
 $p^* = 2$ the easy changes are as  in the proof of Theorem \ref{Partial symmetry}, Step 2).  \\
i) \  If $\eta >0$, choose a compact
$K=K _{\eta } \subset \Sigma_{\lambda _0}^{\nu_0}  $ such that  \
 $
\int_{  \Sigma_{\lambda _0}^{\nu_0}     \setminus K  }u^{p^*}dx =  $
 $
\int_{ ( \Sigma_{\lambda _0}^{\nu_0} \setminus Z_u )   \setminus K  }u^{p^*}dx\leq \frac {\eta }2 $.
 Then   $ \int_{  \Sigma_{\lambda }^{\nu}     \setminus K  }u^{p^*}dx\leq \eta $ will hold not only  if $\lambda $ is sufficiently close to $\lambda _0 $ but also if $\nu$ is sufficiently close to $\nu _0$
  so that
we can fix  $\bar{\varepsilon} >0 $,  such that    \begin{align}\nonumber
K \subset \Sigma_{\lambda }^{\nu} \, , \quad
 \int_{\Sigma_{\lambda }^{\nu}  \setminus K } u^{p^*}dx\leq  \eta \quad \quad  \forall \lambda  \in ( \lambda _0 -\bar  \varepsilon , \lambda _0 +\bar{\varepsilon}) \quad \forall  \nu \in I_{  \varepsilon  }(\nu _0 )
\end{align}
\begin{equation}\nonumber u<u_{\lambda }^{\nu}\quad\hbox{in}\quad K \quad \quad  \forall \lambda  \in ( \lambda _0 -\bar  \varepsilon , \lambda _0 +\bar{\varepsilon}) \quad \forall  \nu \in I_{  \varepsilon  }(\nu _0 )\,.
\end{equation}
Then we proceed as in the previous theorem.
   \par

ii) Let us remark that if $x \in F_{\nu _0} = \{x \in \Sigma _{\lambda _0}^{\nu _0} \setminus Z_{\lambda _0}^{\nu _0} : u(x) = u_{\lambda _0}^{\nu _0}(x)  \} $, then by the strong comparison principle $u\equiv u _{\lambda _0}^{\nu _0} $ in the connected component of $\Sigma _{\lambda _0}^{\nu _0} \setminus Z_{\lambda _0}^{\nu _0}$ to which $x$ belongs, so that  $|\nabla u(x)| = |\nabla u_{\lambda _0}^{\nu _0} (x)| \neq 0  $.
So \ \    $F_{\nu _0} = \{x \in \Sigma _{\lambda _0}^{\nu _0} \setminus Z_u : u(x) = u_{\lambda _0}^{\nu _0}(x)  \} $ \  and we define \ $G_{\nu _0} = \left ( \Sigma^{\nu_0}_{\la_0} \setminus Z^{u} \right ) \setminus  F_{\nu _0} $,
 the complementary set with respect to $ \Sigma _{\lambda _0}^{\nu _0} \setminus Z_u$:  \
   $G_{\nu _0} = \{ x \in \Sigma^{\nu_0}_{\la_0} :u (x)< u^{\nu_0}_{\la_0}(x) , \nabla u(x) \neq 0 \}$.\\
Given $\eta >0$ we first select two compact set
$K^1= K^1 _{\eta } \subset G_{\nu _0}   $ and $K^2= K^2 _{\eta } \subset F_{\nu _0} $ such that putting $K= K _{\eta} = K^1 \cup K^2 $ then  \eqref{compattolambda0}  holds, namely
$
\int_{  \Sigma_{\lambda _0}^{\nu_0}    \setminus K  }u^{p^*}dx\leq \frac {\eta }2   $ and
$
\int_{ ( \Sigma_{\lambda _0}^{\nu_0} \setminus Z_u )   \setminus K  }u^{p^*}dx\leq \frac {\eta }2 $.\\
Since $ u^{\nu_0}_{\la_0}-u $ is positive only in $G_{\nu _0}$, then there exists $\alpha >0 $ such that
$ u^{\nu_0}_{\la_0}-u \geq \alpha >0$ in $K^1 _{\eta }  $, and
 we can fix  $\bar{\varepsilon} >0 $,  such that
   we have
  \begin{align}\label{vvbolita}
K= K^1 \cup K^2 \subset \Sigma_{\lambda }^{\nu} \, , \quad
 \int_{\Sigma_{\lambda }^{\nu}  \setminus K } u^{p^*}dx\leq  \eta \quad \quad  \forall \lambda  \in ( \lambda _0 -\bar  \varepsilon , \lambda _0 +\bar{\varepsilon}) \quad \forall  \nu \in I_{  \varepsilon  }(\nu _0 )
\end{align}
\begin{equation}\nonumber u<u_{\lambda }^{\nu}\quad\hbox{in}\quad K^1     \quad \quad  \forall \lambda  \in ( \lambda _0 -\bar  \varepsilon , \lambda _0 +\bar{\varepsilon}) \quad \forall  \nu \in I_{  \varepsilon  }(\nu _0 )\,.
\end{equation}
 The difference now is that we know that   $ u^{\nu_0}_{\la_0}-u $ is strictly positive only in $K^1 _{\eta } $ and not in $K= K _{\eta} = K^1 \cup K^2 $ .\par
 Nevertheless, if
 $\nu _1 \in I_{  \bar  \varepsilon} (\nu_0)$ and   $
u \; < \; u^{\nu_1}_{\la_1} \quad {\rm in} \quad
F_{\nu _0} \cap \Sigma^{\nu_1}_{\la_1}  \,
$ \ \
then $ u^{\nu_1}_{\la_1}-u  \geq \beta  $ for some $\beta >0$ in $K= K _{\eta} = K^1 \cup K^2 $.
 So  there exists
$  \varepsilon  _1 >0$ ($  \varepsilon  _1 <  \bar  \varepsilon - |\lambda _1 - \lambda _0|$, $  \varepsilon  _1 <  \bar  \varepsilon - |\nu _1 - \nu _0|$) such that
\begin{equation}\label{vvbolita2} u<u_{\lambda }^{\nu}\quad\hbox{in}\quad K     \quad \quad  \forall \lambda  \in ( \lambda _1 -  \varepsilon _1 , \lambda _1 +\varepsilon _1) \quad \forall  \nu \in I_{  \varepsilon _1 }(\nu _1 )\,.
\end{equation}
  Then, \eqref{vvbolita} and \eqref{vvbolita2} (which are  analogous to \eqref{bolita} and \eqref{bolita2}) hold and   we proceed as in Step 2 of Theorem \ref{Partial symmetry} and deduce that  $ \quad u \leq u _{\la}^{\nu } $ in $ \Sigma _{\la}^{\nu } $ \ \ \
$\forall $
$ \la \in  \left ( \la_1 -  \varepsilon _1 \; , \; \la_1 +  \varepsilon _1 \right ) $, \
$ \forall $ $ \nu \in I_{  \varepsilon _1 }(\nu _1 ) $.
  \end{proof}

\smallskip

Once we have a partial symmetry result, the full symmetry follows by using the technique introduced in \cite{DP} for bounded domains and extended in \cite{DPR} (see also \cite{DR}) for unbounded domains. \\
The idea of the proof is to suppose by contradiction that there exists a direction $\nu _0$ and a connected component $C$ of local symmetry different from the whole cap $\Sigma _{\la _0(\nu _0)}^{\nu _0 }  $ (so that $ \partial C \setminus T_{\la _0(\nu _0)}^{\nu _0 } \neq \emptyset $ and there are points $x$ in that boundary where $\nabla u (x) =0 $ ), and simultaneously move hyperplanes orthogonal to different directions close to $\nu _0$.
 In this way we get  a single connected component $C_1$ as in the previous theorem that is a local symmetry component for \emph{all} the directions in a neighborhood of a direction $\nu _1 $ close to $\nu _0$.
Using this component $C$ we can then construct a set $\Gamma \subset \partial C _1$
where $u$ is constant and $\nabla u = 0$, and a ball whose boundary meets $\Gamma $ and  where $u>m$. Using  Hopf's Lemma  we get then a contradiction, since the points on $\Gamma $  are critical points of $u$. \par
We sketch below the details.
Besides Lemma \ref{muoveredirezionielambda} we shall make use of the following simple topological lemma, whose proof can be found in \cite{DP}(Corollary 4.1).
\begin{lemma}\label{topological lemma}
Let $A$ and $B$ be open
connected sets in
a topological space and assume that $A \cap B \neq \emptyset $ and $ A
\not
\equiv B$.
Then, $(\partial A \cap B) \; \cup \; (\partial B \cap A) \; \neq \;
\emptyset $.
\end{lemma}

\bigskip

\begin{proof}[\underline{Proof of Theorem \ref{radialtheorem}}]
Let us now fix a direction $ \nu_0$ such that
$u \not \equiv u^{\nu_{0}}_{\la_0(\nu_{0})}$ in
$\Sigma^{\nu_{0}}_{\la_0(\nu_{0})}$ and  let $ \bar \varepsilon $ be as in Lemma \ref{muoveredirezionielambda} ii) .  \par
Let
${\mathcal{F}}_{\nu_{0}}$ be the collection of the
connected components $C^{\nu_{0}}$ of $(\Sigma^{\nu_{0}}_{\lambda_0 \, (\nu_{0})} \,
\setminus \, Z^{\nu_{0}}_{\lambda_0 \, (\nu_{0})})$ where $u  \equiv u^{\nu_{0}}_{\la_0(\nu_{0})}$.
Since $u \not \equiv u^{\nu_{0}}_{\la_0(\nu_{0})}$ in
$\Sigma^{\nu_{0}}_{\la_0(\nu_{0})}$, there is at least such a component $C$ with $\partial C \setminus T ^{\nu_{0}}_{\la_0(\nu_{0})} \neq \emptyset $, and since the components are open the collection is denumerable.
 Let $\{ C^{\nu_0} _i \} _{i \in I \subset \N}  $ be an enumeration of the sets in
${\mathcal{F}}_{\nu_0}$. \par
We
need to
introduce a symmetrized version $\tilde{\mathcal{F}}_\nu$ of
$\mathcal{F}_\nu $. \par
If $\nabla u(x) = 0 \;$ for all $\; x$ on $ \partial C^\nu,$ then define
$\tilde{C}^\nu = C^\nu.$  If not, there are points $x \in (\partial
C^\nu \,
\cap \, T^\nu_{\lambda_0 \, (\nu)})$ such that $\nabla u(x) \neq 0.$  In
such cases,
define
$$\tilde{C}^\nu = C^\nu \, \cup \, R^\nu_{\lambda_0 \, (\nu)} \,
(C^\nu) \,
\cup \, \{x \in (\partial C^\nu \, \cap \, T^\nu_{\lambda_0 \,
(\nu)}) \, : \, \nabla u (x) \, \neq \, 0 \}.$$
It is easy to see that the sets $\tilde{C}^\nu$ are open and connected, and moreover (and this is the reason to introduce them) we have that
$\nabla u(x)=0$ for any $x \in \partial \tilde{C}^\nu $, while $\nabla u(x)\neq 0$ for any $x \in  \tilde{C}^\nu $. \par
\smallskip
\textbf{STEP 1} \emph{$\lambda_0 \, (\nu)$ is a continuous  function of $\nu$ at $\nu_0$.}
\par
Let us fix $\varepsilon >0 $,  $ \varepsilon < \bar \varepsilon $, where $ \bar \varepsilon $ is as in Lemma \ref{muoveredirezionielambda} ii) .  \par
By the definition of $\lambda_0 \, (\nu_0),$  there exist $\lambda \in
\, (\lambda_0 \,
(\nu_0) , \, \lambda_0 \, (\nu_0)+   \varepsilon) \; \, {\rm and} \,\;  x \in
\Sigma^{\nu_0}_\lambda$ such that $u(x) > u^{\nu_0}_\lambda (x).$  By
continuity of $u$ with respect to $\nu\,$,
there exists $ \delta_1 > 0 $ such that for every $\nu \in I_{\delta_1} (\nu_0) $
$ x \in \, \Sigma^\nu_\lambda$ and $u(x) \, > \, u^\nu_\lambda (x)$.
  Hence,  for every
$\nu \, \in \, I_{\delta_1} \, (\nu_0)$ we have
$ \, \lambda _0 (\nu) \, < \, \lambda_0 \, (\nu_0)  + \varepsilon  $. \par
Now we claim that there exists $ \delta_2 > 0$ such that $\lambda_0(\nu _0) \,- \, \varepsilon <
 \lambda_0 \, (\nu)  \,  $ for any $\nu \in I_{\delta_2} \,
(\nu_0).$
\par
 If this is not true, then there exists a sequence $\{\nu_n\}$ of
directions such that $\nu_n \rightarrow \nu_0 $ and $\lambda_0 \, (\nu_n) \,
\leq \, \lambda_0 \, (\nu_0) \, - \, \varepsilon  \, \; \forall \, \; n.$
By step 1 in Theorem \ref{Partial symmetry}   we have that $\lambda_0 (\nu) \geq - R_0 $ for
any direction $\nu$.  Thus, the sequence $\lambda_0 (\nu_n)$ is bounded and
hence, up to a subsequence, it converges to a number
$\overline{\lambda} \leq \lambda_0 \, (\nu_0) \, - \, \varepsilon $.
  Then, by \eqref{disug.strettaprimadilambda0}
we have
$
u < u^{\nu_0}_{\overline{\lambda}} \,\quad {\rm in} \quad \,
\Sigma^{\nu_0}_{\overline{\lambda}} \; \setminus  \,
Z^{\nu_0}_ {\overline{\lambda}}
$. \par
Now by Lemma \ref{muoveredirezionielambda},  i), \    we have that
$u \leq u^\nu_\lambda$ in
$ \Omega^\nu_\lambda$
for any $\lambda \, $ close to $ \overline \la $
 and any $\nu$ close to $\nu_0$.\par
In particular, this will hold for $\nu_n$ and
$\lambda_0 \, (\nu_n) \, + \, \gamma,$ for some $n$ large and $\gamma$ small,
contradicting the definition of $\lambda_0 \, (\nu_n)$.  \par
 Thus, it follows that
 \begin{equation} \label{deltadiepsilon} \forall \, \varepsilon > 0 \, \exists \, \delta  =  \delta (\varepsilon ) > 0 \, : \quad \la_0(\nu_0) - \varepsilon \, < \, \la_0(\nu) \, < \, \la_0(\nu_0) + \varepsilon \qquad  \forall \,
 \nu \in
I_ \delta (\nu_0)\,.
\end{equation} \par
\smallskip
\textbf{STEP 2}
 \emph{Here we prove that if $0 < \bar \delta  \leq \min \{   \delta \, (\bar \varepsilon ), \bar \varepsilon \}$ (where $\delta \, (\bar \varepsilon )$ is as in \eqref{deltadiepsilon}) then for each
$\nu \in
I_{\bar  \delta  } \, (\nu_0),$  there exists $ i \, \in \, I
 $ such that
$\tilde{C}_i^{\nu_0} \, \in \, \tilde{\mathcal{F}}_\nu$.}\\
 This means that for every direction $\nu $ close to $ \nu _0$, one of the connected components of local symmetry in the direction $\nu _0$ is also a component of local symmetry in the direction $\nu $.\\
The crucial remark that will help us is the following:  if $\nu_1$ and $\nu_2$ are two
directions and  $C^{\nu_1} \in {\mathcal{F}}_{\nu_1}$,
$C^{\nu_2} \in {\mathcal{F}}_{\nu_2}$, then either $\tilde{C^{\nu_1}}
\, \cap \tilde{C^{\nu_2}} = \emptyset$ or $\tilde{C^{\nu_1}} \equiv
\tilde{C^{\nu_2}}$. This is a consequence of Lemma \ref{topological lemma}. In fact, if  $\tilde{C^{\nu_1}}
\, \cap \tilde{C^{\nu_2}} \neq  \phi$ and $\tilde{C^{\nu_1}} \not \equiv
\tilde{C^{\nu_2}}$ then one  set must meet the boundary of the other, and this is impossible since
$\nabla u(x)=0$ for any $x \in \partial \tilde{C}^{\nu _k} $, while $\nabla u(x)\neq 0$ for any $x \in  \tilde{C}^{\nu _k} $, $k=1,2$.  \par
Let now $\nu \in
I_{ \bar \delta  } \, (\nu_0) $,  with $0 < \bar  \delta  \leq \min \{   \delta \, (\bar \varepsilon ), \bar \varepsilon \}$ so that
$|\nu - \nu _0 |$, $|\lambda _0 (\nu) - \lambda _0 (\nu _0 )| < \bar \varepsilon$, and
suppose by contradiction that  $
u \; < \; u^{\nu}_{\lambda _0(\nu )} \quad {\rm in} \quad
F_{\nu _0} \cap \Sigma^{\nu}_{\lambda _0(\nu )}  \,
$ where $F_{\nu _0} = \{ x \in \Sigma^{\nu_0}_{\la_0} \setminus Z^{\nu_0}_{\la_0} : u(x)= u_{\lambda _0}^{\nu _0}(x) \}
=  \cup _{i \in I} C_i^{\nu _0}$.
By Lemma \ref{muoveredirezionielambda} ii)  we would have $u \; < \; u^{\nu}_{\lambda } $ in $\Sigma ^{\nu}_{\lambda } $ for $\lambda > \lambda _0(\nu ) $, contradicting the definition of $\lambda _0(\nu )$.\\
So we have proved that there exists $ i \, \in \, I
 $ ($I$ is the index set of the enumeration of $\tilde{\mathcal{F}}_{\nu _0}$ )  and
$x \in \tilde{C}_i^{\nu_0}$ such that $u(x)= u ^{\nu}_{\lambda _0(\nu )} (x)$. Since $\nabla u(x)\neq 0 $ by the strong comparison principle
$u\equiv  u ^{\nu}_{\lambda _0(\nu )} $ in the component $\tilde {C}^{\nu } $ of $\Sigma^{\nu}_{\lambda _0(\nu )}  \setminus Z ^{\nu}_{\lambda _0(\nu )} $ to which $x$ belongs. Then $\tilde {C}^{\nu }= \tilde{C}_i^{\nu_0}$ by the previous remark (components $\tilde {C^{\nu_1}} \in \tilde {\mathcal{F}}_{\nu_1}$,
$\tilde {C^{\nu_2}} \in \tilde {\mathcal{F}}_{\nu_2}$ either coincide or are disjoint).\par

\smallskip
\textbf{STEP 3} \emph{ There exists a direction $\, \nu ' \,$ near
$\,\nu_0 \,$ and a set ${\tilde{C}}^{\nu_0}_{i} \, \in \,
{\tilde{\mathcal{F}}}_{\nu_{0}} \cap {\tilde{\mathcal{F}}}_{\nu '}$ such that
 the set ${\tilde{C}}^{\nu_0}_{i} \, \in \,
{\tilde{\mathcal{F}}}_\nu$, for \emph{every} direction $\nu$ in a neighbourhood $\overline {I_{ \delta '}( \nu ' )}= \{ \nu : |\nu - \nu ' | \leq \delta '  \}$ of $
\,  \nu '\,$ }.\\
 Let $\{ {\tilde{C_i}}^{\nu_0}\}_{i \in I} $ be an enumeration of the sets in
${\tilde{\mathcal{F}}}_{\nu_0}$, and let $\bar \delta $ be as in Step 2.\\
If  there exists $\delta _0 $, $0 < \delta _0 \leq \bar \delta$ such that ${\tilde{C}}_1^{\nu_0} \in
{\tilde{\mathcal{F}}}_\nu \,$ for all $\, \nu \in \overline {I_{ \delta _0 }
(\nu_0)}$,  then Step 3 is proved, with $ \nu '  = \nu _0$, $\delta ' = \delta _0$ and  we are done.\\
  If not, there exists $\, \nu_1 \in
\overline {I_{\bar \delta  } (\nu_0)}$ such that
${\tilde{C}}_1^{\nu_0} \,\not\in \, {\tilde{\mathcal{F}}}_{\nu_1}$, and
we show now that  $ \, \tilde {C}^{\nu_0}_1 \not\in
\tilde{\mathcal{F}}_\nu \,$ for all the directions  $ \, \nu $ sufficiently close to $\nu _1$. \\
If  $S_1 \subset C_1^{\nu_0} \cap \Sigma _{\lambda _0 (\nu _1)}^{\nu _1}$ is compact, we have that $ u^{\nu_1}_{\la_0(\nu_1)} -u \geq \gamma >0 $
in $S_1$ so that for the directions $\nu $ close to
$ \nu _1 $ we have \  $u < u^{\nu}_{\la_0(\nu)} $ in  $S_1 $, and by the strong comparison principle
  $ \, \tilde {C}^{\nu_0}_1 \not\in
\tilde{\mathcal{F}}_\nu \,$. \\
So there exists  $ \delta _1 \leq
\bar  \delta  - |\nu _1 - \nu _0 | $ such that $ \, \tilde {C}^{\nu_0}_1 \not\in
\tilde{\mathcal{F}}_\nu \,$   for any $ \, \nu \in \overline {I_{ \delta_1}(\nu_1)}$.\\
  Now we check if $\, {\tilde{C}}^{\nu_0}_2 \, \in \,
{\tilde{\mathcal{F}}}_\nu \,$ for all $\, \nu \in \overline {I_{ \delta_1}(\nu_1)}$.\\
 If not, we find a
direction $ \nu _2 \in \overline {I_{ \delta_1}(\nu_1)}  $ and a
neighborhood  $\, \overline  {I_{ \delta_2}(\nu_2)}\,$
such that $\, \delta_2 \leq   \delta _1 - |\nu _2 - \nu _1| \,$ and $ {\tilde{C}}_1^{\nu_0} \, , \,  \,{\tilde{ C}}^{\nu_0}_2 \,\not\in \,
{\tilde{\mathcal F}}_{\nu} \,$ for any $ \, \nu \, \in \, \overline  {I_{ \delta_2}(\nu_2)}$.  \\
Proceeding in this way,
\begin{itemize}
\item[(i)] either we stop at $k$'th stage where for $k \in \ N $ we have
${\tilde{C}}^{\nu_0}_{k+1} \,
\in \, {\tilde{\mathcal F}}_\nu \;
\forall \; \nu \in \overline  {I_{\delta_k} (\nu_k)}$
\item[(ii)] or the process does not stop at all and $I = \N.$
\end{itemize}
Now we claim that (ii) cannot arise. In case (ii), we obtain a sequence
of nested compact sets
$\{\overline{I_{\delta_i}(\nu_i)}\}_{i \in I}$ with the finite intersection
property.  Then, by
Cantor's intersection theorem
$\cap_{i \in I}
\overline{I_{\delta_i}(\nu_i)} \neq \phi.$  For the direction $\nu$ in this
intersection, ${\tilde{C}}_i^{\nu_0}\,\not\in \,
{\tilde{\mathcal{F}}}_\nu$  for any $ i$. But  this
contradicts Step 2, thus Step 3 now follows.\par
\smallskip
\textbf{STEP 4} \emph{ $u  \equiv u^{\nu_{0}}_{\la_0(\nu_{0})}$ in $\Sigma ^{\nu_{0}}_{\la_0(\nu_{0})}$   for any direction $\nu $ and $u$ is radial. }  \par
 Let ${\tilde{C}}^{\nu_1}_{i} \,$ be
as in Step  3 and let  $x_0 \in \partial C \setminus T ^{\nu_1}_{\la_0
(\nu_1)}$, so that $x_0 \cdot \nu _1 < \lambda _0 (\nu _1)$,  $\nabla u(x_0)=0$, and let $x'_0= R ^{\nu_1}_{\la_0
(\nu_1)}(x_0)$. The set  $\Gamma = \{ R ^{\nu}_{\la_0
(\nu )}(x_0')= R ^{\nu}_{\la_0
(\nu)}(R ^{\nu_1}_{\la_0
(\nu_1)}(x_0)) \, : \, \nu \in  \overline {I_{\bar \delta }(\nu _1)}  \}$ is compact, since it is the image of a compact set for  the continuous function
 $\nu \in \overline {I_{\bar \delta }(\nu _1)} \mapsto  R ^{\nu}_{\la_0
(\nu )}(x_0') $. \par
  By definition,  $\Gamma \subset Z_u$ and $u=m=u(x_0)=u(x_0')$ is constant on $\Gamma $. \par
Moreover the projection of $\Gamma $ on the
hyperplane $T^{\nu _1}_{\lambda_0 \, (\nu _1)}$ contains an $N-1$-dimensional ball $B_{\R^{N-1}}(z_0, \alpha)$ open in
$T^{\nu _1}_{\lambda_0  \, (\nu _1)}$.\\
This is impossible by Proposition 5.1 in \cite{DPR} which state precisely that such a set cannot exist.
For the reader's convenience we give here a simple proof in our case, assuming for simplicity that the direction is $\nu _1 = e_1$ and $x_0  = (y_0, z_0)$, with
$y_0 \in \R $, $z_0 \in \R ^{N-1}$, $y_0 < \lambda _0 = \lambda _0 (e_1) $, and
$x_0'= (y_0', z_0)= (2 \lambda _0 - y_0, z_0)$ where $\lambda _0 = \lambda _0 (e_1)$. \par
Observe  that since the function $\lambda _0 (\nu )$ is continuous, taking $\bar \delta $ smaller if necessary, the set $\Gamma $ is a (compact) subset of the half space
$\Sigma _{\lambda _0}=\Sigma _{\lambda _0 (e_1)}^{e_1} $, it has a positive distance $d$ from the hyperplane  $T_{\lambda _0}=T_{\lambda _0 (e_1)}^{e_1} $, and if  $m$ is the constant value of $u$ on $\Gamma $,
  by the monotonicity of $u$ in the cap $\Sigma _{\lambda _0 }$, we have that $u \geq m $
  in the part $R_{\Gamma }$ of $\Sigma _{\lambda _0  } $  to the right  of $\Gamma $, i.e. in the segments parallel to the direction $ x _1$ going from each point $ R ^{\nu}_{\la_0
(\nu)}( x_0')$,  $ \nu \in  \overline {I_{\bar \delta}(e _1)}  $, to the hyperplane $T_{\lambda _0  } $. \par
If $d= \min \{ \text{ dist }\left (\Gamma , T_{\lambda _0} \right ) \,,\, \alpha \} >0$ (where the projection of $\Gamma $ on the
hyperplane $T^{\nu _1}_{\lambda_0 \, (\nu _1)}$ contains an $N-1$-dimensional ball $B_{\R^{N-1}}(z_0, \alpha)$ open in
$T^{\nu _1}_{\lambda_0  \, (\nu _1)}$), we can move the balls $B_{\lambda } = B\left ( (\lambda ,z_0), \frac d2 \right )$ decreasing $\lambda $ from $\lambda _0 $ until $\overline {B_{\lambda } }$ first touches $\Gamma$ in a point $\bar x \in \Gamma \cap \partial B_{\lambda }$. This will happen for
$\bar \lambda \leq \lambda _0 - \frac d2 $
and a ball $B=B_{\bar \lambda}= B \left ( (\bar \lambda ,z_0), \frac d2 \right ) \subset  R_{\Gamma } \subset \Sigma _{\lambda _0}$. \par
 Since $- \Delta _p (u-m) = u^{p^*-1} >0 $ in $B$ and $u - m \geq 0 $ in $B$, we have that $u>m $ in $B$ and by Hopf's Lemma we have $\nabla u (\bar x) \neq 0 $, and this contradicts  the fact that $\Gamma $ is a subset of $Z_u$.\par
So it is impossible that ii) holds in Theorem \ref{Partial symmetry} and this means that for any direction $\nu $ we have $u \equiv u_{\lambda _0 (\nu)}^{\nu}$ in $\Sigma _{\lambda _0 (\nu)}^{\nu}$. \par

  Furthermore we can conclude using the same method that
$Z_u \cap \Sigma^\nu_{\lambda_0  \, (\nu)} \, = \,  \emptyset $ for any
direction $\nu $.\par
 Thus,   $u$ is
strictly increasing in  every direction $\nu$ in  $ \Sigma^\nu_{\lambda_0  \, (\nu)}$, with $u \equiv u_{\lambda _0 (\nu)}^{\nu}$ in $\Sigma _{\lambda _0 (\nu)}^{\nu}$.  \par
Considering now $N$
linearly independent directions in $\R^N$, the final symmetry
result follows:
 $u$ is radially symmetric about some point $x_0 \in
\R^N$ \  ($x_0= \cap _{k=1}^N T_{\lambda _0 (e_k)}^{e_k}$) \ , which is the only critical point of $u$,  and it is  strictly radially decreasing.
 \par
   \end{proof}

\section{Proof of Theorem \ref{triscorojfmdmgmgmgbmbxbdkjgdfkghk} and Theorem \ref{corojfmdmgmgmgbmbxbdkjgdfkghk}}\label{sectioncorta}
The proof of Theorem \ref{triscorojfmdmgmgmgbmbxbdkjgdfkghk} and the proof of Theorem \ref{corojfmdmgmgmgbmbxbdkjgdfkghk} are direct consequences of Theorem \ref{radialtheorem} and the results in \cite{willem}.\\
 Let us only provide a few details for the reader's convenience.

\begin{proof}[\underline{Proof of Theorem \ref{triscorojfmdmgmgmgbmbxbdkjgdfkghk}}]
Once that Theorem \ref{radialtheorem} is proved,
 the proof follows  directly by Theorem 1.2 in \cite{willem} since the solutions to the critical problem in the whole space are now classified
 by \eqref{classificationveron}.
\end{proof}

\begin{proof}[\underline{Proof of Theorem \ref{corojfmdmgmgmgbmbxbdkjgdfkghk}}]
Let us first observe that we are in position to apply Theorem \ref{triscorojfmdmgmgmgbmbxbdkjgdfkghk}. \\
\emph{Let us consider the case $(i)$.}\\
It is standard to see that, if $v_0$ is not trivial, then
\[
J(v_0)>\frac{S^{\frac Np}}{N}\,.
\]
This follows   using $v_0$ as a test function in the equation $-\Delta_p v_0 =v_0^{p^*-1}$ in $\Omega$, and then, exploiting the Sobolev's inequality, recalling that the best Sobolev's constant is not achieved in bounded domains.
Therefore, $(i)$ follows by \eqref{trisssprojectionenetimateenery} recalling that
\[
J_\infty (\mathcal{U}_{\lambda_i,x_i})=\frac{S^{\frac Np}}{N}\,.
\]
\noindent \emph{Let us now consider the case $(ii)$.}\\
We deduce that   $k=1$ in this case exploiting again the  fact that $J_\infty (\mathcal{U}_{\lambda_i,x_i})=\frac{S^{\frac Np}}{N}$ and \eqref{trisssprojectionenetimateenery}.  Then, we set $\lambda_1=\lambda$ and $x_1=x_0$ and $(ii)$ is proved.\\
Finally, exploiting again the fact that $J_\infty (\mathcal{U}_{\lambda_i,x_i})=\frac{S^{\frac Np}}{N}$, we deduce
 that $c=\frac{S^{\frac Np}}{N}$. Namely, if  $S^{\frac{N}{p}}/N < c<2S^{\frac{N}{p}}/N$, necessarily  $(i)$ occurs.
\end{proof}

\end{document}